\documentclass[11pt]{article}

\usepackage{enumerate}
\usepackage{amsmath}
\usepackage{amssymb,latexsym}
\usepackage{amsthm}
\usepackage{color}
\usepackage{graphicx}
\usepackage{amscd}
\usepackage{hyperref}

\title{On scattered linear sets of pseudoregulus type in $\PG(1,q^t)$}
\author{Bence Csajb\'ok \and Corrado Zanella}

\newcommand{\cB}{{\mathcal B}}
\newcommand{\cT}{{\mathcal T}}
\newcommand{\cC}{{\mathcal C}}

\newcommand{\cF}{{\mathcal F}}
\newcommand{\cP}{{\mathcal P}}
\newcommand{\cH}{{\mathcal H}}

\newcommand{\Lb}{{\mathbb L}}

\newcommand{\cQ}{{\mathcal Q}}
\newcommand{\cS}{{\mathcal S}}

\newcommand{\F}{{\mathbb F}}
\newcommand{\Fq}{\F_q}
\newcommand{\Fqt}{\F_{q^t}}

\newcommand{\la}{\langle}
\newcommand{\ra}{\rangle}

\newcommand{\Qed}{\hfill $\Box$ \medskip}

\newcommand{\hs}{{\hat{\sigma}}}

\newtheorem{theorem}{Theorem}[section]

\newtheorem{lemma}[theorem]{Lemma}
\newtheorem{corollary}[theorem]{Corollary}
\newtheorem{definition}[theorem]{Definition}
\newtheorem{proposition}[theorem]{Proposition}

\newtheorem{remark}[theorem]{Remark}

\DeclareMathOperator{\PG}{{PG}}

\DeclareMathOperator{\PGL}{{PGL}}
\DeclareMathOperator{\PGaL}{P\Gamma L}
\DeclareMathOperator{\Gal}{Gal}

\begin{document}
\maketitle

\begin{abstract}
Scattered linear sets of pseudoregulus type in $\PG(1,q^t)$ have been defined and investigated in
\cite{LuMaPoTr2014,DoDu9}.
The aim of this paper is to continue such an investigation.
Properties of a scattered linear set of pseudoregulus type, say $\Lb$, 
are proved by means of three different ways to obtain $\Lb$:
(i) as projection of a $q$-order canonical subgeometry  \cite{LuPo2004},
(ii) as a set whose image under the field reduction map is the hypersurface of degree $t$ in $\PG(2t-1,q)$
studied in  \cite{LaShZa2013},
(iii) as exterior splash,  by the correspondence described in \cite{LaZa2015}.
In particular, given a canonical subgeometry $\Sigma$ of $\PG(t-1,q^t)$,
necessary and sufficient conditions are given for the projection of $\Sigma$ 
with center a $(t-3)$-subspace to be a linear set of pseudoregulus type.
Furthermore, the $q$-order sublines are counted and geometrically described.
\end{abstract}



\section{Introduction}

If $V$ is a vector space over  the finite field $\F_{q^t}$, 
then $\PG_{q^t}(V)$ denotes the projective space whose points are the one-dimensional $\F_{q^t}$-subspaces of $V$. 
If $V$ has dimension $n$ over $\F_{q^t}$, then $\PG_{q^t}(V)=\PG(n-1,q^t)$. 
For a point set $T\subset \PG(n-1,q^t)$ denote by $\la T \ra$ the projective subspace 
of $\PG(n-1,q^t)$ spanned by the points in $T$. 
For $m \mid t$ and a set of elements $S\subset V$ denote by $\la S \ra_{q^m}$ the $\F_{q^m}$-vector subspace 
of $V$ spanned by the vectors in $S$. 
For the rest of the paper 
assume that $q=p^e$ is a power of the prime $p$.
Also $\theta_s=(q^{s+1}-1)/(q-1)$ for $s\in\mathbb N\cup\{-1\}$, and $N(x)=q^{\theta_{t-1}}$ is the norm
of $x\in\Fqt$ over $\Fq$.

Let $R=\F_{q^t}^r$, and let
$P=\PG_{q^t}(T)$ be a point of $\PG(r-1,q^t)$, where $T$ is a one-dimensional $\F_{q^t}$-subspace of $R$. 
Then $\cF_{r,t,q}(P):=\PG_q(T)$ defines the \emph{field reduction} from $\PG(r-1,q^t)$ to $\PG(rt-1,q)$. 
Denote the point set of $\PG(r-1,q^t)$ by $\cP$. 
Then $\cF_{r,t,q}(\cP)$ is a Desarguesian $(t-1)$-spread of $\PG(rt-1,q)$ \cite{Lu1999}. 
Let $S$ be a subspace of $\PG(rt-1,q)=\PG_q(R)$, then 
\[ \cB(S):=\{P \in \PG(r-1,q^t) \mid \cF_{r,t,q}(P)\cap S \neq \emptyset\}.\]
A point set $L \subseteq \PG(r-1,q^t)$ is said to be \emph{$\F_q$-linear}
(or just \emph{linear}) \emph{of rank $n$} if $L=\cB(S)$ for some $(n-1)$-subspace $S\subseteq \PG(rt-1,q)$.
The size of such $L$ is at most $\theta_{n-1}$; if the size of $L$ is equal to $\theta_{n-1}$, then $L$ is a
\emph{scattered} linear set.

\begin{definition}
Let $\PG_{q^t}(V)=\PG(n-1,q^t)$, let $W$ be an $n$-dimensional $\F_q$-vector subspace of $V$, 
and $\Sigma=\{ \la w \ra_{q^t} \mid w\in W^*\}$. 
If $\la \Sigma\ra=\PG(n-1,q^t)$, then $\Sigma$ is a 
\emph{($q$-order) canonical subgeometry} of $\PG(n-1,q^t)$.
\end{definition}

Let $\Sigma$ be a $q$-order canonical subgeometry of ${\overline{\Sigma}}=\PG(n-1,q^t)$. 
Let $\Gamma \subset {\overline{\Sigma}} \setminus \Sigma$  be an $(n-1-r)$-space 
and let $\Lambda \subset {\overline{\Sigma}} \setminus \Gamma$ be an $(r-1)$-space of ${\overline{\Sigma}}$. 
The projection of $\Sigma$ from {\it center} 
$\Gamma$ to {\it axis} $\Lambda$ is the point set
\begin{equation}
\label{proj}
L=p_{\,\Gamma,\,\Lambda}\left(\Sigma\right):=\{\la \Gamma, P \ra \cap \Lambda \mid P\in \Sigma\}.
\end{equation}

In \cite{LuPo2004} Lunardon and Polverino characterized linear sets as projections of canonical subgeometries. 
They proved the following.

\begin{theorem}[{\cite[Theorems 1 and 2]{LuPo2004}}]
\label{LuPo}
Let ${\overline{\Sigma}}$, $\Sigma$, $\Lambda$, $\Gamma$ and $L=p_{\,\Gamma,\,\Lambda}(\Sigma)$ be defined as above. 
Then $L$ is an $\F_q$-linear set of rank 
$n$ and $\la L \ra=\Lambda$. 
Conversely, if $L$ is an $\F_q$-linear set of rank $n$ of $\Lambda=\PG(r-1,q^t)\subset {\overline{\Sigma}}$ 
and $\la L \ra=\Lambda$, then there is an $(n-1-r)$-space $\Gamma$ disjoint from $\Lambda$ 
and a $q$-order canonical subgeometry $\Sigma$ disjoint from  $\Gamma$ 
such that $L=p_{\,\Gamma,\,\Lambda}(\Sigma)$.
\end{theorem}

Note that when $r=n$ in Theorem \ref{LuPo}, then $L=p_{\,\Gamma,\,\Lambda}(\Sigma)=\Sigma$, 
hence $L$ is a canonical subgeometry.

A family of scattered $\F_q$-linear sets of rank $tm$ of $\PG(2m-1,q^t)$, called of {\it pseudoregulus type}, have been introduced in \cite{MaPoTr2007} for $m=2$ and $t=3$, further generalized in \cite{LaVa2013} for $m\geq 2$ and $t=3$ and finally in \cite{LuMaPoTr2014} for $m\geq 1$ and $t\geq 2$ (for $t=2$ they are the same as Baer subgeometries isomorphic to $\PG(2m-1,q)$, see \cite[Remark 3.4]{LuMaPoTr2014}). 
This paper is devoted to the investigation of linear sets of pseudoregulus type in $\PG(1,q^t)$. 
It has been proved in \cite[Section 4]{LuMaPoTr2014} and in \cite[Remark 2.2]{DoDu9} that all linear sets of pseudoregulus type in $\PG(1,q^t)$ are $\mathrm{PGL}(2,q^t)$-equivalent and hence they ca ben defined as follows. 

\begin{definition}[\cite{LuMaPoTr2014,DoDu9}]
\label{pseudo}
A point set $\Lb$ of $\PG_{q^t}(\F_{q^t}^2)=\PG(1,q^t)$, $t\geq 2$, is called a \emph{linear set of pseudoregulus type} if $\Lb$ is projectively equivalent to
\begin{equation}
\label{psedefi}
\Lb_0=\{\la (\lambda,\lambda^q)\ra_{q^t} \mid \lambda\in \F_{q^t}^*\}.
\end{equation}
If $\phi$ is a projectivity mapping $\Lb_0$ into $\Lb$, 
then $\la(1,0)\ra_{q^t}^\phi$ and $\la(0,1)\ra_{q^t}^\phi$ are  \emph{transversal points of} $\Lb$.
\end{definition}
There are precisely two transversal points \cite[Proposition 4.3]{LuMaPoTr2014}.

Section \ref{s:center} is devoted to the characterization (Theorem \ref{t:main}) 
of the projecting configuration giving rise,
according to Theorem \ref{LuPo}, to a linear set of pseudoregulus type in $\PG(1,q^t)$, say $\Lb$.
In particular, a necessary and sufficient condition is that the center of the projection
is of type \[\Gamma=\la P,P^\hs,\ldots,P^{\hs^{t-3}}\ra,\] where $P$ is an imaginary point
with respect the canonical subgeometry $\Sigma$, and $\hs$ is a collineation of order $t$
fixing $\Sigma$ pointwise.
Dualizing Theorem \ref{t:main} allows a description of all line and canonical subgeometry pairs
giving rise to an exterior splash which is a linear set of pseudoregulus type in $\PG(1,q^t)$
(Theorem \ref{t:splash}).

In Section \ref{s:NRC} a sufficient condition (Proposition \ref{t:NRC-reali}) is proved for a 
normal rational curve in $\PG(t-1,q^t)$, $t$ prime, to be $\Fq$-rational.

In Section \ref{s:sublines} the $q$-order sublines contained in a
linear set of pseudoregulus type in $\PG(1,q^t)$ are thoroughly investigated.
Their number is computed (\ref{e:numero-sublines}) for $q\ge t$.
A particular subset of them is described, where the $q$-order sublines are projections of normal $\Fq$-rational
curves of order $t-1$ containing the basis $\{P,P^\hs,\ldots,P^{\hs^{t-1}}\}$.
This generalizes the result on $q$-order sublines in \cite[Theorem 5.2]{BaJa2014}.

In order to describe all $q$-order sublines in $\Lb$, in
Section \ref{theta} the notion of a $d$-power of a line is introduced and investigated.
Assuming $\PG(t-1,q)=\PG_q(\Fqt)$,
the $d$-power $\ell^d$ of a line $\ell$ is the set of all $\la x^d\ra_q$ for $\la x \ra_q\in\ell$.
In Theorem \ref{maintheta} a sufficient condition on $d$ is proved for $\ell^d$ to be a normal rational curve.
As a consequence, 
for $q\ge t$ any $q$-order subline in $\Lb$ is projection of a normal rational curve (Theorem \ref{t:tutte-NRC}).
Additionally, if $t$ is prime, there are $(t-1)$ families of sublines arising from normal rational curves
of all orders from one to $t-1$, where the order is constant inside any family (Theorem \ref{p:unique-family}).

\section{Characterization of the projecting configurations}
\label{s:center}
In this section $\Lb$ denotes a scattered linear set of pseudoregulus type in $\ell_0\cong\PG(1,q^t)$, $t\geq 3$. 
It is assumed that for a $(t-3)$-subspace $\Gamma$ of $\PG(t-1,q^t)\supset\ell_0$,
  with $\Gamma\cap\ell_0=\emptyset$,
  and a $q$-order $(t-1)$-dimensional subgeometry $\Sigma$, $\Sigma\cap\Gamma=\emptyset$,
  of $\PG(t-1,q^t)$ it holds
  $p_{\Gamma,\ell_0}(\Sigma)=\Lb$.
    In $\PG(t-1,q^t)$ take homogeneous coordinates such that $\Sigma$ is the set of points having
  coordinates $x_1, x_2,\ldots,x_t$ in $\F_q$.

  Let $\hat\Gamma$ and $\hat\ell_0$ be the $(t-2)$-subspace and 2-subspace of $\F_{q^t}^t$
  related to $\Gamma$ and $\ell_0$, respectively.
  Two vectors $v_1,v_2\in\hat\ell_0$ exist such that
  \begin{equation}\label{e:char_L}
    \Lb=\{\langle \lambda v_1+\lambda^q v_2\rangle_{q^t}\mid\lambda\in\F_{q^t}^*\}.
  \end{equation}
  Denote by $\varphi_i$, $i=1,2$, the canonical projection onto $\langle v_i\rangle_{q^t}$
  with respect to $(\hat\Gamma+\langle v_{3-i}\rangle_{q^t})\oplus\langle v_i\rangle_{q^t}$.
  By this definition 
  $\Lb=\{\langle a^{\varphi_1}+a^{\varphi_2}\rangle_{q^t}\mid a\in(\F^t_{q})^*\}$.
  This implies that letting
  \[
    \varphi_i(a_1,a_2,\ldots,a_t)=\sum_{j=1}^t\mu_{ij}a_jv_i,\quad i=1,2,
  \]
  both $(\mu_{11},\mu_{12},\ldots,\mu_{1t})$ and $(\mu_{21},\mu_{22},\ldots,\mu_{2t})$
  are $t$-tuples of $\F_q$-linearly independent elements of $\F_{q^t}$.
  Then the map $\varphi:\,\F_{q^t}\rightarrow\F_{q^t}$ defined by
  $(\sum_{j=1}^t\mu_{1j}a_j)^\varphi=\sum_{j=1}^t\mu_{2j}a_j$ for any $a_1$, $a_2$, $\ldots$, $a_t\in\F_q$
  is a well-defined non-singular $\F_q$-linear map.
  The linear set $\Lb$ can be expressed by means of this map as follows:
  \begin{equation}\label{e:rechar_L}
     \Lb=\{\langle \lambda v_1+\lambda^{\varphi} v_2\rangle_{q^t}\mid\lambda\in\F_{q^t}^*\}.
  \end{equation}
\begin{proposition}\label{p:phi}
The function $\varphi$ has the following properties.
\begin{enumerate}[(i)]  
  \item Any $\lambda\in\F_{q^t}^*$ satisfies $N(\lambda^\varphi\lambda^{-1})=1$. 
  \item For any $\beta,\lambda\in\F_{q^t}$, $\lambda\neq0$, it holds $(\beta\lambda)^\varphi=\beta\lambda^\varphi$
  if and only if $\beta\in\F_q$.
\end{enumerate}
\end{proposition}
\begin{proof}
  The first statement follows by comparing (\ref{e:char_L}) and (\ref{e:rechar_L}).
  Taking into account (\ref{e:rechar_L}) once again, we have that the size $\theta_{t-1}$ of $\Lb$ equals
  the size of the image of
  $\psi:\,\F_{q^t}^*\rightarrow\{z\in\F_{q^t}\mid N(z)=1\}$
  defined by $\lambda^\psi=\lambda^\varphi\lambda^{-1}$.
  Since all non-zero elements of a one-dimensional $\Fq$-subspace have the same image,
  the images of any two $\F_q$-linearly independent vectors are distinct.
  This implies that for any nonzero $\lambda,\mu\in\F_{q^t}$,
  \[
    \langle\lambda\rangle_q=\langle\mu\ra_q\ \Leftrightarrow\ 
    \lambda^\varphi\lambda^{-1}=\mu^\varphi\mu^{-1}\ \Leftrightarrow\ 
    \frac{\mu^\varphi}{\lambda^\varphi}=\frac{\mu}{\lambda}.
  \]
  The thesis follows by applying the last equivalences to $\mu=\beta\lambda$.
\end{proof}
In the main result of this section (Theorem \ref{t:main}) the following fact 
on permutation polynomials, i.e., polynomials
in $\F_q[x]$ which are bijective maps, is needed.
\begin{theorem}\label{t:MC1963}\emph{\cite{Ca1960,MC1963}}
  Let $F$ be a field of order $q'=p^n$, 
	and\footnote{The condition $d>1$ is not explicitly stated in\cite{MC1963}, but is clear from the context. See also \cite{BrLe1973}.} 
  $1<d\mid q'-1$, say $q'-1=dm$.
  Define $\Psi_d(x)=x^m$ for $x\in F$.
  If $f(x)$ is a permutation polynomial on $F$ satisfying $f(0)=0$, $f(1)=1$, 
  and $\Psi_d\{f(x)-f(y)\}=\Psi_d(x-y)$ for all $x,y$, then
  an integer  $j$ satisfying $d\mid p^j-1$, $0\le j<n$ exists such that $f(x)=x^{p^j}$
  for all $x\in F$.
\end{theorem}

\begin{theorem}\label{t:main}
  Let $\Sigma$ be a $q$-order canonical subgeometry of $\PG(t-1,q^t)$, $q>2$, $t\ge3$.
  Assume that $\Gamma$ and $\ell_0$ are a $(t-3)$-subspace and a line of $\PG(t-1,q^t)$, respectively,
  such that $\Sigma\cap\Gamma=\emptyset=\ell_0\cap\Gamma$.
  Then the following assertions are equivalent:
  \begin{enumerate}[(i)]
  \item   
    The set $p_{\Gamma,\ell_0}(\Sigma)$ is a scattered  $\F_q$-linear set of pseudoregulus type.
  \item
    A generator $\hs$ exists of the subgroup of $\PGaL(t,q^t)$ fixing pointwise $\Sigma$, 
    such that $\dim(\Gamma\cap\Gamma^\hs)=t-4$; furthermore, $\Gamma$ is not contained in the
    span of any hyperplane of $\Sigma$.
  \item
    There are a point $P_\Gamma$ and a generator $\hs$ of the subgroup of $\PGaL(t,q^t)$ 
    fixing pointwise $\Sigma$, such that 
    $\langle P_\Gamma,P^\hs_\Gamma,\ldots,P^{\hs^{t-1}}_\Gamma\rangle=\PG(t-1,q^t)$, and
    \begin{equation}\label{e:Gamma}
      \Gamma=\langle P_\Gamma,P^\hs_\Gamma,\ldots,P^{\hs^{t-3}}_\Gamma\rangle.
    \end{equation}
  \end{enumerate}
  Furthermore, if the conditions above are satisfied, then
  \begin{enumerate}[(a)]
  \item For a fixed $\hs$, the point $P_\Gamma$ satisfying (\ref{e:Gamma}) is unique;
  \item the transversal points are precisely $\ell_0\cap\la\Gamma,P_\Gamma^{\hs^i}\ra$, $i=t-2,t-1$.
  \end{enumerate}
\end{theorem}
\begin{proof}
  (i) $\Rightarrow$ (ii).
  Theorem \ref{t:MC1963} for $f=\varphi/(1^\varphi)$, $q'=q^t$ and $d=q-1$ together with Proposition\ 
  \ref{p:phi} imply that an $\alpha\in\F_{q^t}$ with $N(\alpha)=1$ and an integer $\nu$ with $\gcd(t,\nu)=1$ exist
  such that
  $\lambda^\varphi=\alpha \lambda^{q^\nu}$ for any $\lambda\in\F_{q^t}$.
  The map 
  \begin{equation}\label{e:nu}
    \sigma:x\mapsto x^{q^\nu}
  \end{equation}
  induces a generator $\hs$ 
  of the subgroup of $\PGaL(t,q^t)$ fixing pointwise $\Sigma$, and $\Gamma$ is intersection of
  the hyperplanes $H$ of equation
  $\sum_{j=1}^t\mu_{1j}X_j=0$ and $H^\hs$ of equation
  $\sum_{j=1}^t\mu_{1j}^\sigma X_j=0$.
  This implies that $\Gamma\cap\Gamma^\hs=H\cap H^\hs\cap H^{\hs^2}$ has dimension at least $t-4$.
  If $\Gamma\cap\Gamma^\hs$ had dimension $t-3$, that is $\Gamma=\Gamma^\hs$,
  then $\Gamma$ would be the span of a subspace of the canonical subgeometry $\Sigma$
  (see e.g. \cite[Lemma (3.2)]{Lu1984}), contradicting the assumption that
  $p_{\Gamma,\ell_0}(\Sigma)$ is scattered.
  By the same assumption, 
  $\Gamma$ is not contained in the span of any hyperplane of $\Sigma$.
  
  (ii) $\Rightarrow$ (iii).
  Since  $\Sigma=\{\la a\ra_{q^t}\mid a\in(\F_q^t)^*\}$ is the standard canonical subgeometry,
  the map $\hs$ can be identified with an automorphism $\sigma$ of order $t$ of $\F_{q^t}$.
  Let $H^\hs$ be the hyperplane spanned by $\Gamma$ and $\Gamma^\hs$.
  Hence $\Gamma\subseteq H\cap H^\hs$.
  The equation $H=H^\hs$ would imply that $H$ is the span of a hyperplane of $\Sigma$, 
  a contradiction, so $\Gamma= H\cap H^\hs$.
  It holds
  \[
    \bigcap_{j=0}^{t-3}\Gamma^{\hs^j}=\bigcap_{j=0}^{t-2}H^{\hs^j}\neq\emptyset.
  \]
  If $\dim\left(\bigcap_{j=0}^{t-3}\Gamma^{\hs^j}\right)>0$, then 
  $\bigcap_{j=0}^{t-1}H^{\hs^j}\neq\emptyset$, implying that both $H$ and $\Gamma$ contain a
  point of $\Sigma$, a contradiction.
  So, $\bigcap_{j=0}^{t-3}\Gamma^{\hs^j}$ is one point, say 
  \begin{equation}\label{e:Q}
    Q=\bigcap_{j=0}^{t-2}H^{\hs^j}.
  \end{equation} 
	Since for $i\in\{0,1,\ldots,t-1\}$ the point $Q^{\hs^i}$ is the intersection of $t-1$ hyperplanes 
	in the independent set $\{H,H^{\hs},\ldots,H^{\hs^{t-1}}\}$, the set of points $\{Q,Q^{\hs},\ldots,Q^{\hs^{t-1}}\}$ is independent. 
  Define $P_\Gamma=Q^{\hs^3}$. 
  Because of (\ref{e:Q}), for $i=0,1,\ldots,t-3$ it holds $P_\Gamma^{\hs^i}\in\Gamma$. 
  So, (\ref{e:Gamma}) is true.  
  
  (iii) $\Rightarrow$ (i).
  Let $P_\Gamma=\la(b_1,b_2,\ldots,b_t)\ra_{q^t}$.
  The line $\ell_0$ such that $\ell_0\cap\Gamma=\emptyset$ is immaterial, 
  so assume $\ell_0=\la P_\Gamma^{\hs^{t-2}},P_\Gamma^{\hs^{t-1}}\ra$.
  The projectivity $\kappa\in\PGL(t,q^t)$ related to the matrix
  \begin{equation}
	\label{e:kappa}
  A=\begin{pmatrix} b_1&b_1^\sigma&\ldots&b_1^{\sigma^{t-1}}\\ 
  b_2&b_2^\sigma&\ldots&b_2^{\sigma^{t-1}}\\ \vdots&&&\\
  b_t&b_t^{\sigma}&\ldots&b_t^{\sigma^{t-1}}\end{pmatrix}^{-1}
  \end{equation}
  maps $\ell_0$ and $\Gamma$ into the subspaces of equations $X_1=X_2=\ldots=X_{t-2}=0$ and $X_{t-1}=X_t=0$,
  respectively.
  Note that the points of the canonical subgeometry $\{\la (\lambda,\lambda^\sigma,\ldots,\lambda^{\sigma^{t-1}})\ra_{q^t} \mid \lambda\in \F_{q^t}^*\}$ are mapped by $\kappa^{-1}$ into points of $\Sigma$, so by a cardinality argument 
  it holds $\Sigma^\kappa=\{\la (\lambda,\lambda^\sigma,\ldots,\lambda^{\sigma^{t-1}})\ra_{q^t} \mid \lambda\in \F_{q^t}^*\}$. 
  Observing that
  \[
   p_{\Gamma^\kappa,{\ell_0}^\kappa}(\Sigma^\kappa)=
   \{\la(0,\ldots,0,\mu,\mu^\sigma)\ra_{q^t}\mid \mu\in\F_{q^t}^*\}
  \]
  is a scattered linear set of pseudoregulus type completes the first part of the proof.
We remark here that 
\begin{equation}\label{e:tau}
    \kappa^{-1}\hs\kappa:\,\la(x_1,x_2,\ldots,x_t)\ra_{q^t}
    \mapsto \la(x_t^\sigma,x_1^\sigma,\ldots,x_{t-1}^\sigma)\ra_{q^t}.
  \end{equation}. 

\begin{sloppypar}
  If a point $R$ satisfies $\la R, R^\hs,\ldots,R^{\hs^{t-1}}\ra=\PG(t-1,q^t)$ and 
  $\la R, R^\hs,\ldots,R^{\hs^{t-3}}\ra=\Gamma$,
  then $\bigcap_{i=0}^{t-3}\Gamma^{\hs^i}=R^{\hs^{t-3}}$, that is $R=P_\Gamma$.
  This proves (a).
\end{sloppypar}
  
  \begin{sloppypar}
  As regards (b), note that the 
  transversal points of $p_{\Gamma^\kappa,{\ell_0}^\kappa}(\Sigma^\kappa)$
  are $\la(0,\ldots,0,1,0)\ra_{q^t}$ and $\la(0,\ldots,0,1)\ra_{q^t}$. Applying $\kappa^{-1}$ it follows 
  that the transversal points of $p_{\Gamma,\ell_0}(\Sigma)$ are $P_\Gamma^{\hs^{t-2}}$ and $P_\Gamma^{\hs^{t-1}}$, respectively.
  \end{sloppypar}
\end{proof}

\begin{remark}
  The property $\dim(\Gamma\cap\Gamma^{\hs})=t-4$ in Theorem \ref{t:main} does not hold in general
  for any $\hs$.
  As a matter of fact, if $\gcd(s,t)=1$, then the projection of the canonical subgeometry 
	$\{\la(\lambda,\lambda^{\hs},\ldots,\lambda^{\hs^{t-1}})\ra_{q^t}\mid\lambda\in\F_{q^t}^*\}$ from
  the center $\Gamma_s$ of equation
  $X_{t-s}=X_{t}=0$ is a scattered linear set of pseudoregulus type but, if $1<s<t-1$,
  $\Gamma_s\cap\Gamma_s^{\kappa^{-1}\hs\kappa}$ is a $(t-5)$-subspace.
\end{remark}

\begin{remark}
\label{r:uniqueness}
  By Theorem \ref{t:main} (b), there are precisely two collineations $\hs$ satisfying the conditions 
  (i), (ii), (iii). They are inverse of each other.
\end{remark}

The \emph{splash} of a $q$-order canonical subgeometry $\Sigma$ of $\PG(t-1,q^t)$ on a line $\ell$ is the set of all intersections of $\ell$ 
(not contained in the span of a hyperplane of $\Sigma$) with the spans of the hyperplanes of $\Sigma$, and is always a linear set \cite{LaZa2015}. The relationship between tangent splashes and linear sets has been dealt with in \cite{LaZa2014_3}.
By dualizing Theorem \ref{t:main} one obtains a characterization of
the linear sets of pseudoregulus type among all exterior splashes.
\begin{theorem}\label{t:splash}
  Let $\ell$ be a line in $\PG(t-1,q^t)$, exterior to a $q$-order canonical subgeometry $\Sigma$ and 
	not contained in the span of a hyperplane of $\Sigma$. 
  The splash $\Lb$ of $\Sigma$ on $\ell$ is a scattered $\F_q$-linear set of pseudoregulus type if and only if
  a generator $\hs$ exists of the subgroup of $\PGaL(t,q^t)$ fixing pointwise $\Sigma$, such that
  $\ell\cap\ell^\hs\neq\emptyset$.
  In this case the transversal points of $\Lb$ are $P=\ell\cap\ell^\hs$ and $P'=\ell\cap\ell^{\hs^{-1}}$. \Qed
\end{theorem}
For $t=3$ the carriers defined in \cite{BaJa2014} are the transversal points \cite[Theorem 4.2]{BaJa2014}.

\section{\texorpdfstring{Normal $\Fq$-rational curves}{Normal Fq-rational curves}}\label{s:NRC}

In this section we collect results that will be used in the proof of Theorem \ref{t:subl-v-carriers},
but can have their own interest.
Fix a generator $\sigma$ of $\Gal(\Fqt/\Fq)$ ($t\ge3$); it induces an element $\hs$ of $\PGaL(t,q^t)$.
It holds
\begin{proposition}\label{p:immaginario}\emph{\cite[Lemma 3.51]{FiFi}}
  Let $P=\la(\alpha_1,\alpha_2,\ldots,\alpha_t)\ra_{q^t}$ be a point of $\PG(t-1,q^t)$.
  The following conditions are equivalent:
  \begin{enumerate}[(i)]
    \item
    $\la P,P^\hs,\ldots,P^{\hs^{t-1}}\ra=\PG(t-1,q^t)$;
    \item 
    $\alpha_1,\alpha_2,\ldots,\alpha_t$ are $\Fq$-linearly independent;
    \item
    \begin{equation}\label{e:det-imm}
      \det\begin{pmatrix}\alpha_1&\alpha_2&\ldots&\alpha_t\\
      \alpha^\sigma_1&\alpha^\sigma_2&\ldots&\alpha^\sigma_t\\
      \vdots&&&\vdots\\
      \alpha_1^{\sigma^{t-1}}&\alpha_2^{\sigma^{t-1}}&\ldots&\alpha_t^{\sigma^{t-1}}\end{pmatrix}\neq0.
    \end{equation}
  \end{enumerate}
\end{proposition}
A point $P$ satisfying the conditions in Proposition\ \ref{p:immaginario} will be called an 
\emph{imaginary point}.
From now on $\PG(t-1,q)$ is considered to be the set of all $\Fq$-rational points of $\PG(t-1,q^t)$.
\begin{proposition}
  If $\tau\in\PGL(t,q)$ and $P$ is an imaginary point in $\PG(t-1,q^t)$, then $P^\tau$ is an
  imaginary point. \Qed
\end{proposition}
\begin{proposition}\label{p:t+2pts}
  Assume $t$ is prime.
  Let $Q_1$ and $Q_2$ be two distinct $\Fq$-rational points in $\PG(t-1,q^t)$.
  If $P$ is an imaginary point, then no $t$ of the points $Q_1$, $Q_2$, $P$, $P^\hs$, $\ldots$,
  $P^{\hs^{t-1}}$ lie on a hyperplane.
\end{proposition}
\begin{proof}
  Let $S$ be a $t$-set contained in $\{Q_1,Q_2,P,P^\hs,\ldots,P^{\hs^{t-1}}\}$.
  If $S$ does not contain $\Fq$-rational points, then the thesis follows from the definition of an imaginary point.
  If $S$ contains exactly one $\Fq$-rational point, say $Q_1$, then $H=\la S\setminus\{Q_1\}\ra$ 
  is a hyperplane spanned by $t-1$ of
  the points $P$, $P^\hs$, $\ldots$, $P^{\hs^{t-1}}$, so $H$, $H^\hs$, $\ldots$, 
  $H^{\hs^{t-1}}$
  are independent points of the dual space; in particular, their intersection is empty.
  If $Q_1\in H$, then $Q_1\in H\cap H^\hs\cap\ldots\cap H^{\hs^{t-1}}$, a contradiction.

  It only remains to prove that if $Q_1,Q_2\in S$, then $S$ is independent.
  Let $r$ and $r+s$, $0\le r<r+s\le t-1$, be the indices such that $P^{\hs^r}$, $P^{\hs^{r+s}}\not\in S$.
  Assume $S$ is a dependent set of points.  
  Like in the proof of Theorem \ref{t:main}, take homogeneous coordinates $x_1$, $x_2$, $\ldots$, $x_t$ 
  such that $P$, $P^\hs$, $\ldots$, $P^{\hs^{t-1}}$
  are base points (i.e. all their coordinates but one are equal to zero), 
  with first base point $P^{\hs^r}$, 
  and where $\hs$ acts as in \eqref{e:tau}. 
   If $P^{\hs^r}\in\la S\ra$, then
  $Q_1,Q_2\in\la\{P,P^\hs,\ldots,P^{\hs^{t-1}}\}\setminus\{P^{\hs^{r+s}}\}\ra$, 
  and one obtains a contradiction as above.
  Therefore, the hyperplane $K=\la S\ra$ has equation $X_1+\lambda X_{1+s}=0$
  for some $\lambda\in\Fqt$.
  Since $Q_1,Q_2\in K\cap K^\hs\cap\ldots\cap K^{\hs^{t-2}}$, the equations
  \begin{equation}\label{e:iperpiani}
    X_i+\lambda^{\sigma^{i-1}} X_{i+s}=0,\quad i=1,2,\ldots,t-1
  \end{equation}
  are linearly dependent (indices are taken modulo $t$).
  Let $I=\{i_1,i_2,\ldots,i_u\}$ be the set of indices of a minimal linearly dependent set of
  equations in (\ref{e:iperpiani}).
  The multiset
  \[
    \bigcup_{k=1}^u\{i_k,i_k+s\}
  \]
  contains each integer mod $t$ either zero or two times, hence $i$ belongs to $I$ if and
  only if $i+s\in I$.
  From $u<t$ it follows $s>1$, and $s$ is a nontrivial divisor of $t$.
\end{proof}

\begin{remark}
  The assumption that $t$ is a prime cannot be removed from the previous proposition.
  Indeed, assume $P$ is an imaginary point in $PG(3,q^4)$.
  The line $\ell=\la P,P^{\hs^2}\ra$ is fixed by $\hs^2$, so its intersection
  with $\PG(3,q^2)$ is a line in the latter projective space.
  It is well known that there are precisely $q^2+1$ $\Fq$-rational lines, forming a spread of $\PG(3,q)$,
  which intersect $\ell$.
\end{remark}
  
Let $q\ge t-1$.
A \emph{normal rational curve of order $t-1$} in $\PG(t-1,q^t)$ is any $\cC_t^\tau$, where
\begin{equation}\label{e:cC_i}
  \cC_i=\{\la(1,y,\ldots,y^{t-1})\ra_{q^t}\mid y\in\mathbb F_{q^i}\}\cup\{\la(0,0,\ldots,0,1)\ra_{q^t}\}\quad
  (i\mid t),
\end{equation}
and $\tau\in\PGL(t,q^t)$.
When $\tau\in\PGL(t,q)$, $\cC_t^\tau$ is a \emph{normal $\Fq$-rational curve}.
\begin{theorem}\label{t:HT}\emph{\cite[Theorem 27.5.1 (v)]{JWPH3}}
  Let $q\ge t+1$.
  Then there is a unique normal rational curve of order $t-1$ in $\PG(t-1,q)$ through any $t+2$ points no
  $t$ of which lie in a hyperplane.
\end{theorem}
A normal $\Fq$-rational curve can be characterized by means of the number of its $\Fq$-rational points.
\begin{proposition}
  Let $\cC$ be a normal rational curve of order $t-1$ in $\PG(t-1,q^t)$.
  Assume $q\ge t+1$.
  Then $\cC$ is $\Fq$-rational if and only if $\cC$ has $q+1$ $\Fq$-rational points.
\end{proposition}
\begin{proof}
  For any $\tau\in\PGL(t,q)$, a point $P^\tau$ is $\Fq$-rational if and only if $P$ is $\Fq$-rational.
  This implies that the number of $\Fq$-rational points of a normal $\Fq$-rational curve is exactly $q+1$.
  
  Conversely, assume that $\cC$ has at least $q+1$ $\Fq$-rational points.
  Fix a $(t+2)$-subset, say $\cS$, of $\cC\cap\PG(t-1,q)$.
  By Theorem \ref{t:HT} a normal rational curve $\cC_1^\tau$, $\tau\in\PGL(t,q)$ exists containing $\cS$.
  Then $\cC$ and $\cC_t^\tau$ share at least $t+2$ points.
  By Theorem \ref{t:HT} again, $\cC=\cC_t^\tau$.
\end{proof}

\begin{proposition}\label{t:NRC-reali}
  Assume $t$ is prime, and $q\ge t+1$.
  Let $Q_1$ and $Q_2$ be two distinct $\Fq$-rational points in $\PG(t-1,q^t)$.
  If $P$ is an imaginary point, then a unique normal rational curve
  $\cC\subset\PG(t-1,q^t)$ exists which contains $Q_1$, $Q_2$, $P$, $P^\hs$, $\ldots$,
  $P^{\hs^{t-1}}$, and such $\cC$ is $\Fq$-rational.
\end{proposition}
\begin{proof}
  Proposition \ref{p:t+2pts} and $q^t\ge t+1$ are the assumptions of Theorem \ref{t:HT} which states the
  existence and uniqueness of $\cC$.
  
  The number of imaginary points in $\PG(t-1,q^t)$ is
  \[
    K_1=q^{t-1}\prod_{i=0}^{t-2}(q^{t-1}-q^i).
  \]
  Since $t$ is prime, the number of imaginary points of a normal $\Fq$-rational curve is $K_2=q^t-q$.
  Let $K_3$ be the constant number of normal $\Fq$-rational curves containing two distinct $\Fq$-rational points.
  By double counting the number of triples $(R_1,R_2,\cS)$, where $R_1$ and $R_2$ are $\Fq$-rational points,
  $R_1\neq R_2$, and $\cS$ is a normal $\Fq$-rational curve containing both $R_1$ and $R_2$,
  taking into account that the total number of normal $\Fq$-rational curves in $\PG(t-1,q^t)$ is
  \cite[Theorem 27.5.3 (ii)]{JWPH3}
  \[
    \nu_{t-1}=\frac{\prod_{i=0}^{t-1}(q^t-q^i)}{q(q^2-1)(q-1)},
  \]
  one obtains
  \[
    \frac{q(q^t-1)(q^{t-1}-1)}{(q-1)^2}K_3=(q+1)q\nu_{t-1},
  \]
  whence
  \[
    K_3=\frac{\prod_{i=0}^{t-1}(q^t-q^i)}{q(q^t-1)(q^{t-1}-1)}.
  \]
  Now let $M$ be the number of pairs $(X,\cT)$, where $X$ is an imaginary point, and $\cT$ is a normal
  $\Fq$-rational curve containing $X$ and the given points $Q_1$, $Q_2$.
  There is at most one such curve, for it also contains $X^\hs$, $X^{\hs^2}$, $\ldots$, $X^{\hs^{t-1}}$.
  Then $K_1\ge M=K_2K_3$, and the equality holds if and only if for any imaginary point $X$
  the unique normal rational curve containing $Q_1$, $Q_2$, $X$, $X^{\hs}$, $\ldots$, $X^{\hs^{t-1}}$
  is $\Fq$-rational.
  A direct computation shows that indeed $K_1=K_2K_3$.  
\end{proof}

\section{\texorpdfstring{$q$-order sublines}{q-order sublines}}
\label{s:sublines}

\begin{proposition}
\label{p:kollin}
Let $\Sigma$ and $\Gamma$ be a pair of a canonical subgeometry and a $(t-3)$-space of $\PG(t-1,q^t)$, respectively, 
such that $p_{\Gamma,\ell_0}(\Sigma)$ is a scattered linear set of pseudoregulus type for some line $\ell_0$, $\ell_0\cap \Gamma = \emptyset$. 
Then, up to projectivities, 
\begin{enumerate}[(i)]
	\item $\Sigma=\{\la(\lambda,\lambda^\sigma,\ldots,\lambda^{\sigma^{t-1}})\ra_{q^t} 
	\mid \lambda \in \F_{q^t}^*\}$,
	\item $\Gamma$ is the hyperplane with equations $X_{t-1}=X_t=0$,
	\item the projection of $\Sigma$ from $\Gamma$ onto the line $\ell_1$ with equations 
	$X_1=X_2=\ldots=X_{t-2}=0$ is 
	\[\{\la (0,0,\ldots,0,\mu,\mu^{\sigma})\ra_{q^t} \mid \mu \in \F_{q^t}^*\},\]
\end{enumerate}
where $\sigma$ is a field automorphism $\sigma\colon x\mapsto x^{q^\nu}$ with $\gcd(\nu,t)=1$.
\end{proposition}
\begin{proof}
Note that (iii) is a simple consequence of the first two conditions. 

A projectivity $\alpha$ exists such that $\Sigma^\alpha$ is as in Section \ref{s:center}, 
that is, $\Sigma^\alpha$ consists of the $\F_q$-rational points of $\PG(t-1,q^t)$. 
Let $\sigma \colon x \mapsto x^{q^\nu}$, $\gcd(\nu,t)=1$, be a field automorphism such that 
\[\hs \colon \la(x_1,x_2,\ldots,x_t)\ra_{q^t} \mapsto  \la(x_1^{\sigma},x_2^{\sigma},\ldots,x_t^{\sigma})\ra_{q^t}\]
is one of the two collineations (cf. Remark \ref{r:uniqueness}) satisfying the conditions of Theorem \ref{t:main} (i),\,(ii),\,(iii). 
Then there is a unique point $P_{\Gamma}$ such that $\Gamma=\la P_{\Gamma}, P_{\Gamma}^{\hs},\ldots,P_{\Gamma}^{\hs^{t-3}}\ra$ and 
$P_{\Gamma}$ is an imaginary point of $\PG(t-1,q^t)$ (cf. Theorem \ref{t:main} (b)). 
Consider the projectivity $\kappa$, related to $P_{\Gamma}$ and defined in \eqref{e:kappa}. 
Then $\alpha\kappa$ will be good for the purpose of the proof. 
\end{proof}

It follows from \cite{CsZa2015} that for $t=5$ or $t>6$ Proposition \ref{p:kollin} does not hold anymore if one fixes $\nu$.
For there is no collineation fixing $\Gamma$ and mapping 
$\{\la(\lambda,\lambda^{q},\ldots,\lambda^{q^{t-1}})\ra_{q^t}\mid\lambda \in \F_{q^t}^*\}$ 
 into 
$\{\la(\lambda,\lambda^{q^\zeta},\ldots,\lambda^{q^{\zeta(t-1)}})\ra_{q^t}\mid\lambda \in \F_{q^t}^*\}$
if $\zeta\neq 1,t-1$.

For the rest of this section, we may assume that 
coordinates are fixed such that conditions (i) and (ii) are satisfied in Proposition \ref{p:kollin}. 

\begin{definition}
Let $\iota \colon \PG_q(\F_{q^t}) \rightarrow \Sigma$ be the projectivity
\begin{equation}
\label{e:iota}
\la \lambda \ra_q \mapsto \la (\lambda^{\sigma^2},\lambda^{\sigma^3},\ldots, \lambda^{\sigma^{t-1}},\lambda,\lambda^\sigma) \ra_{q^t}.
\end{equation}
\end{definition}

Since the line $\ell_0$ such that $\ell_0\cap \Gamma=\emptyset$ is immaterial, we may assume that 
$\ell_0$  is the line $\ell_1$ in Proposition \ref{p:kollin}.
Take $x_{t-1}$ and $x_t$ as homogeneous coordinates in $\ell_0=\ell_1$, so
\begin{equation}\label{e:L-2c}
  \Lb=\{\la(\mu,\mu^\sigma)\ra_{q^t}\mid\mu\in\Fqt^*\}.
\end{equation}
The field reduction $\cF=\cF_{2,t,q}$ maps any point $X=\la(a,b)\ra_{q^t}$ of $\ell_0$ to the 
$(t-1)$-subspace $\cF(X)=\{\la(za,zb)\ra_q\mid z\in \F_{q^t}^*\}$ of $\PG(2t-1,q)\cong\PG_q(\Fqt^2)$.

\subsection{\texorpdfstring{Number of $q$-order sublines}{Number of q-order sublines}}

In \cite{LaShZa2013} the following hypersurface of degree $t$ in $\PG(2t-1,q)$ is dealt with:
\begin{equation}\label{e:cq}
  \cQ_{t-1,q}=\{ \la(a,b)\ra_{q}\mid(a,b)\in (\Fqt^2)^*,\ N(a)=N(b)\}.
\end{equation}
The $(t-1)$-subspaces of type
\begin{equation}\label{e:def-Sh}
  S_{h,k}=\{\la(z,kz^{q^h})\ra_q\mid z\in \Fqt^*\}\ \mbox{for $k\in \Fqt$, $N(k)=1$, $h=0,1,\ldots,t-1$,}
\end{equation}
are contained in $\cQ_{t-1,q}$, and any family
\[
  \cS_h=\{S_{h,k}\mid k\in \Fqt,\ N(k)=1\},\quad h=0,1,\ldots,t-1,
\]
is a partition of $\cQ_{t-1,q}$ \cite{LaShZa2013}.
If $q\ge t$, any subspace of $\PG(2t-1,q)$ which is contained in $\cQ_{t-1,q}$ 
is contained in some $S_{h,k}$  \cite[Corollary 10]{LaShZa2013}.

\begin{proposition}\label{p:f-di-l}
  It holds $\cF(\Lb)=\cS_0$.
\end{proposition}
\begin{proof}
  Any $X\in\Lb$ is of type $\la(1,\mu^{q-1})\ra_{q^t}$, $\mu\in \Fqt^*$, that is 
  $X=\la(1,k)\ra_{q^t}$ for some $k\in \Fqt$, $N(k)=1$.
  Hence $\cF(X)=\{\la(z,kz)\ra_q\mid z\in \Fqt^*\}=S_{0,k}$.
  The statement follows from the fact that both $\cF(\Lb)$ and $\cS_0$ are of size $\theta_{t-1}$.
\end{proof}

As a consequence of Proposition\ \ref{p:f-di-l}, it holds
\begin{proposition}
  Let $q\ge t$.
  For any $q$-order subline $r$ in $\Lb$, there is a line $m$ of $\PG(2t-1,q)$ contained in
  some $S_{h,k}$, $h=1,2,\ldots,t-1$, $N(k)=1$, such that $m$ is a transversal line to the regulus $\cF(r)$.
\end{proposition}

\begin{proposition}\label{p:lemma-numero}
  Assume $q\ge t$.
  Let $P=\la(1,1)\ra_{q}$ and, for any $y\in\Fqt\setminus\Fq$ and $h\in\{0,1,\ldots,t-1\}$,
  define $Q_{y,h}=\la(y,y^{q^h})\ra_q$. Then
  \begin{enumerate}[(i)]
  \item The points $P$ and $Q_{y,h}$ in $\PG(2t-1,q)$ are distinct.
  \item For any point $Q$ in $\cQ_{t-1,q}\setminus\{P\}$, the line $\la P,Q\ra$ is contained in $\cQ_{t-1,q}$
  if and only if there are $y\in\Fqt\setminus\Fq$ and $h\in\{0,1,\ldots,t-1\}$ such that
  $Q=Q_{y,h}$.
  \item
  For any two pairs $\{y,h\}$ and $\{y',h'\}$, $y,y'\in \F_{q^t}\setminus \F_q$ and $h,h'\in \{0,1,\ldots,t-1\}$, 
  it holds 
	$Q_{y,h}=Q_{y',h'}$ if and only if $\la y \ra_q=\la y'\ra_q$ and $[\F_q(y) : \F_q]$ divides $h-h'$. 
  \end{enumerate}
\end{proposition}
\begin{proof}
  The `only if' part of assertion (ii) follows from the fact that for $q\ge t$, any line of $\cQ_{t-1,q}$ is contained in some $S_{h,k}$
  and $\la(1,1)\ra_q \in S_{h,k}$ if and only if $k=1$. The `if' part follows from the fact that 
	$\la\lambda(1,1) + (y,y^{q^h})\ra_q =\la(\lambda+y,(\lambda+y)^{q^h})\ra_q\in S_{h,1}$ for any $\lambda\in \F_q$. (In general, for any two 
	points of $S_{h,k}$ the line joining them is also contained in $S_{h,k}$). 
		
  As regards (iii), we may assume $h\geq h'$. We have $Q_{y,h}=Q_{y',h'}$ if and only if 
	$\lambda y = y'$ and $\lambda y^{q^h} = y'^{q^{h'}}$ for some $\lambda \in \F_q^*$. 
	If this happens, then clearly $\la y \ra_q = \la y' \ra_q$ and $y^{q^{h-h'}-1} = 1$, i.e $y\in \F_{q^{h-h'}}^*$ and hence $[\F_q(y) : \F_q] \, |\, h-h'$. On the other hand, if $[\F_q(y) : \F_q] \, |\, h-h'$ for some $y\in\F_{q^t}\setminus \F_q$, then $y\in \F_{q^{h-h'}}$ and hence 
	$y^{q^h}=y^{q^{h'}}$. It follows that for any $y' \in \la y \ra_q$, i.e. when $y'= \lambda y$ for some $\lambda \in \F_q^*$, we have $\lambda y^{q^h}=y'^{q^{h'}}$.
\end{proof}

\begin{proposition}
Denote by $N_1$ the number of lines contained in $\cQ_{t-1,q}$ and incident with $P$. If $q\ge t$, then
  \[
    N_1=\frac{\sum\limits_{y\in\Fqt\setminus\Fq}[\Fq(y):\Fq]}{q(q-1)}.
  \]
\end{proposition}
\begin{proof}
\begin{sloppypar}
  According to Proposition \ref{p:lemma-numero} (iii), for any $y\in \F_{q^t}\setminus \F_q$ 
  and $h,h'\in \{0,1,\ldots,t-1\}$, we have $Q_{y,h}=Q_{y,h'}$ if and only if $h\equiv h' {\pmod {[\F_q(y) : \F_q]}}$. 
  Also, if $Q_{y,h}=Q_{y',h'}$, then $\la y \ra_q=\la y'\ra_q$. 
	Then the statement follows from a double counting 
  of the distinct pairs $(y,Q_{y,h})$, $y\in\Fqt\setminus\Fq$, $h\in\{0,1,\ldots,t-1\}$. 
  The number of such pairs is 
	$\sum\limits_{y\in\Fqt\setminus\Fq}[\Fq(y):\Fq]=N_1 q (q-1)$. 
\end{sloppypar}\end{proof}
\begin{proposition}
  If $q\ge t$, the number of lines of $\PG(2t-1,q)$ which are contained in $\cQ_{t-1,q}$ is
  \[
    N_2=\frac{\theta_{t-1}^2\sum\limits_{y\in\Fqt\setminus\Fq}[\Fq(y):\Fq]}{q(q^2-1)}.
  \]
\end{proposition}
\begin{proof}
  Since the collineation group of $\cQ_{t-1,q}$ is transitive, see \cite[Proposition 19]{LaShZa2013}, every point in $\cQ_{t-1,q}$ is on
  precisely $N_1$ lines, and the size of $\cQ_{t-1,q}$ is $\theta_{t-1}^2$.
\end{proof}
\begin{theorem}\label{t:numero-sublines}
  If $q\ge t$, the number of $q$-order sublines in $\Lb$ is
  \begin{equation}\label{e:numero-sublines}
    \frac{\theta_{t-1}}{q+1}\left(\frac{\sum\limits_{y\in\Fqt\setminus\Fq}[\Fq(y):\Fq]}{q(q-1)}
    -\theta_{t-2}\right).
  \end{equation}
\end{theorem}
\begin{proof}
  This is the number $N$ of lines in $\cQ_{t-1,q}$ which are not contained in any element of $\cS_0$,
  \[
    N=N_2-\frac{\theta_{t-1}^2\theta_{t-2}}{q+1},
  \]
  divided by the number of lines in $\PG(2t-1,q)$ related to a common $q$-order subline in
  $\PG(1,q^t)$, which is $\theta_{t-1}$.
\end{proof}

The following is a corollary of Theorem \ref{t:numero-sublines}:

\begin{proposition}
  Let $q\ge t$.  
  If $t$ is prime, then $\Lb$ contains precisely $(t-1)\theta_{t-1}\theta_{t-2}/\theta_1$
  $q$-order sublines.
\end{proposition}

\subsection{\texorpdfstring{Structure of the $q$-order sublines}{Structure of the q-order sublines}}
\label{s:strutture}

\begin{definition}
\label{order}
Let $P=\la x \ra_q$ be a point of $\PG_q(\F_{q^t})$. We define $o(P)$, the order of $P$, as the smallest integer $m$, such that $x \in \F_{q^m}$. Instead of $o(P)=o(\la x \ra_q)$, we also write $o(x)$. Let $\ell$ be a line of $\PG_q(\F_{q^t})$, and let $\hat\ell$ be the corresponding two-dimensional $\F_q$-subspace of $\F_{q^t}$. We define $o(\ell)$, the order of $\ell$,  as the smallest integer $m$ such that $\hat\ell$ is contained in a one-dimensional $\F_{q^m}$-subspace of $\F_{q^t}$. 
Note that if $\la x \ra_q$ and $\la y \ra_q$ are two different points of $\ell$, then $o(y/x) = o(\ell)$. In particular, if $\la 1 \ra_q \in \ell$, then for each $\la z \ra_q \in \ell \setminus \la 1 \ra_q$ we have $o(z)=o(\ell)$. 
\end{definition}

\begin{proposition}
\label{p:inverse}
Consider $\Lb=p_{\Gamma,\ell_0}(\Sigma)=\{\la(\mu,\mu^\sigma)\ra_{q^t} \mid \mu \in \F_{q^t}^*\}$ with 
$\sigma \colon x \mapsto x^{q^{\nu}}$ (cf. (\ref{e:L-2c})) in 
$\PG(t-1,q^t)$, $q\geq t$. Then for each $q$-order subline $r$ contained in $\Lb$, $\Sigma\cap p^{-1}_{\Gamma,\ell_0}(r)$ is projectively equivalent to the set of all points $\la x^{\delta\theta_{h-1}}\ra_q$, where $x$ varies on a line $\ell$ of $\PG_q(\Fqt)$, $\delta$ is an integer such that
  \begin{equation}\label{e:d-theta}
    \delta\theta_{\nu-1}\equiv 1 {\pmod {\theta_{t-1}}}
  \end{equation}
  and $h$ is such that a transversal line $f$ to the regulus $\cF(r)$ is contained in a
  subspace of $\cS_h$. Also, $o(\ell)$ does not divide $h$. 
\end{proposition}
\begin{proof}
  The line $f\subseteq S_{h,k}$ contains two distinct points in the form $\la(z,kz^{q^h})\ra_q$, $\la(w,kw^{q^h})\ra_q$, thus 
	the points of $r$ have form  
	\[\la(az+bw,akz^{q^h}+bkw^{q^h})\ra_{q^t}=\la(1,k(az+bw)^{q^h-1})\ra_{q^t},\] 
	where $(a,b)\in (\Fq^2)^*$. Take any $u\in \F_{q^t}^*$ and denote $\la u \ra_q \in \PG_q(\F_{q^t})$ by $U$. 
	Then $p_{\Gamma,\ell_0}(U^\iota)=\la(1,u^{q^\nu-1})\ra_{q^t}$  (cf.\ \eqref{e:iota}). 
	Thus $p_{\Gamma,\ell_0}(U^\iota) \in r$ if and only if $u^{q^\nu-1}=k(az+bw)^{q^h-1}$, that is, when 
  \[
    u^{q-1}=k^{\delta}\left[(az+bw)^{q^h-1}\right]^{\delta},
  \]
	for some $(a,b)\in (\Fq^2)^*$. Since $k=k'^{q-1}$ for some $k'\in \F_{q^t}^*$ and $u$ is defined up to a non-zero factor in $\Fq$,
  the point set of $\PG_q(\F_{q^t})$ mapped into $r$ by $\iota p_{\Gamma,\ell_0} $ is, 
  up to the collineation $\la x \ra_q \mapsto \la k'^{\delta}x \ra_q$, equals
		\[
    \{\la(a z+ b w)^{\delta\theta_{h-1}}\ra_q \mid (a,b)\in (\Fq^2)^*\}.
  \]
	Since $\iota$ is a projectivity, the first part follows with $\hat\ell=\la z,w \ra_q$. 
	Let $o(\ell)=o(z/w)=m$ and suppose to the contrary $m \mid h$. Then $\la x^{\delta\theta_{h-1}}\ra_q = \la y^{\delta\theta_{h-1}}\ra_q$, 
	for any two points $\la x \ra_q ,\la y \ra_q \in \ell$, a contradiction. 
\end{proof}

\begin{theorem}\label{t:subl-v-carriers}
  Let $t$ be prime and $q\geq t+1$. 
  There are precisely $\theta_{t-1}\theta_{t-2}/\theta_1$ $q$-order sublines of $\Lb$ which are projections
  under $p_{\Gamma,\ell_0}$ of normal rational curves of order $t-1$ containing the points
  $P_{\Gamma}^{\hs^i}$, $i=0,1,\ldots,t-1$.
\end{theorem}
\begin{proof}
  Since a normal rational curve of order $t-1$ 
  is uniquely determined by $t+2$ points in general position, 
  by Proposition\ \ref{p:t+2pts} any  
  two distinct
  points in $\Sigma$ are contained in precisely one such curve,
  which is $\Fq$-rational (cf.\ Proposition \ref{t:NRC-reali}), and a double counting shows that the number
  of curves is as stated.
  So it is enough to prove, extending the argument in \cite{LaVa2010},  
  that for any normal $\Fq$-rational curve $\cC$ of order $t-1$ in $\PG(t-1,q^t)$
  and any imaginary point $P\in\cC$,
  the projection of $\cC$ from the center $\Gamma=\la P,P^\hs,\ldots,P^{\hs^{t-3}}\ra$
  is a $q$-order subline.
  
  Up to a change of coordinates with coefficients in $\Fq$, $\cC=\cC_t$ (see (\ref{e:cC_i})), so
  $P=\la v\ra_{q^t}$, where $v=(1,\alpha,\alpha^2,\ldots,\alpha^{t-1})$,
  for an $\alpha$ such that $\Fq(\alpha)=\Fqt$.
  The vectors $v$, $v^{\sigma}$, $\ldots$, $v^{\sigma^{t-1}}$
  are $\Fq$-linearly independent.
  The axis of the projection is immaterial, so it can be identified with the line $\ell_1$ of equations
  $X_1=X_2=\ldots=X_{t-2}=0$ that by Proposition \ref{p:t+2pts} is disjoint with $\Gamma$.
  Note that by operating row reduction on the matrix
  \[
    A=\begin{pmatrix}1&x&x^2&\ldots&x^{t-1}\\ 1&\alpha&\alpha^2&\ldots&\alpha^{t-1}\\
    1&\alpha^{\sigma}&\alpha^{2\sigma}&\ldots&\alpha^{(t-1)\sigma}\\ \vdots&&&&\vdots\\
    1&\alpha^{\sigma^{t-3}}&\alpha^{2\sigma^{t-3}}&\ldots&\alpha^{(t-1)\sigma^{t-3}}\end{pmatrix}
  \]
  the last line is $(0,0,\ldots,1,\det A_{t}^{-1}\det A_{t-1})$, where
  $A_j$ is obtained from $A$ by deleting its $j$-th column, $j=t-1,t$.
  The matrix $A_t$ is a Vandermonde matrix, and
  by a general property of Vandermonde matrices,
  \[
    \det A_{t-1}=\left(x+\sum_{i=0}^{t-3}\alpha^{\sigma^i}\right)\det A_t.
  \]
  This implies that
  \[
    p_{\Gamma,\ell_1}(\cC\cap\Sigma)=
    \{\la(0,0,\ldots,0,1,x+\sum_{i=0}^{t-3}\alpha^{\sigma^i})\ra_{q^t}\mid x\in\Fq\}
    \cup\{\la(0,0,\ldots,0,1)\ra_{q^t}\}
  \]
  is a $q$-order subline.
\end{proof}
\begin{remark} Theorem \ref{p:unique-family} will state that the $q$-order sublines dealt with 
in Theorem \ref{t:subl-v-carriers} are all associated with a unique family $\cS_h$.
\end{remark}

\section{\texorpdfstring{$\theta_{\nu-1}^{-1} \theta_{h-1}$-powers of lines }{Powers of lines}}\label{theta}

\begin{definition}
For an integer $d$ and a point set $\cH \subseteq \PG_q(\F_{q^t})$, 
let $\cH^d=\{ \la x^d \ra_q \mid \la x \ra_q \in \cH\}$. 
\end{definition}

\begin{definition}
For any $w\in \F_{q^t}^*$ let $\lambda_w : \PG_q(\F_{q^t}) \rightarrow \PG_q(\F_{q^t})$ denote the projectivity $\la z \ra_q \mapsto \la wz \ra_q$.
\end{definition}

First we show three simple properties of the $\la x \ra_q \mapsto \la x^d \ra_q$ map. Proposition \ref{prop 3} says that 
the study of $\ell^d$ can be reduced to the study of $\ell^{d'}$, where $d\equiv d' {\pmod {\theta_{m-1}}}$, $m$ the order of $\ell$. 
Take a line $\ell$ of $\PG_q(\F_{q^t})$. When $t$ is not a prime, and hence $\F_{q^t}$ has non-trivial subfields, then Proposition \ref{prop 2} says that the dimension of $\la \ell^d \ra$ depends also on the choice of $\ell$ and not only on the choice of $d$. Proposition \ref{projeq} implies that the study of $\ell^d$ can be reduced to the study of $\ell'^d$, where $\ell'$ is a line such that $o(\ell)=o(\ell')$ and $\ell'$ contains $\la 1 \ra_q$.

\begin{proposition}
\label{prop 3}
If $\ell$ is a line of $\PG_q(\F_{q^t})$ with $o(\ell)=m$ and $d \equiv d' \pmod {\theta_{m-1}}$, then 
$\ell^d$ is $\PGL(t,q)$-equivalent with $\ell^{d'}$. 
\end{proposition}
\begin{proof}
Let $x=\alpha f_1$ and $y=\alpha f_2$ for some $f_1, f_2 \in \F_{q^m}$ and $\alpha\in \F_{q^t}$ such that $\ell=\la x,y \ra_q$. 
We show that an element $\beta$ of $\F_{q^t}^*$ exists such that $\lambda_{\beta}(P^d)=P^{d'}$ for each $P\in \ell$. 
Take a point $P\in \ell$, so $P=\la \eta x + \mu y \ra_q$ for some $\eta, \mu \in \F_q$. 
Then $\lambda_{\beta}(P^d)=P^{d'}$ if and only if $\beta^{q-1}(\eta x + \mu y)^{d(q-1)}=(\eta x + \mu y)^{d'(q-1)}$, that is,
$\beta^{q-1}=(\eta x + \mu y)^{(d'-d)(q-1)}$. We have $d-d'=K\theta_{m-1}$ for some integer $K$, thus 
$(\eta x + \mu y)^{(d-d')(q-1)}=(\eta x + \mu y)^{(q^m-1)K}=\alpha^{(q^m-1)K}$. 
It follows that $\alpha^{\theta_{m-1}K}$ is a good choice for $\beta$. 
\end{proof}

\begin{proposition}
\label{prop 2}
Let $\ell$ be a line of $\PG_q(\F_{q^t})$ with $o(\ell)=m$. 
Then $\ell^d$ is contained in a subspace $\PG(m-1,q)$ of $\PG_q(\F_{q^t})$. 
Also, there is a line $\ell'$ in $\PG_q(\F_{q^m})$ with $o(\ell)=o(\ell')$, such that $\ell^d$ is $\PGL(t,q)$-equivalent with $\ell'^d$. 
\end{proposition}
\begin{proof}
We have $\ell=\la x,y \ra_q$ for some $x=\alpha f_1$ and $y=\alpha f_2$, where $f_1, f_2 \in \F_{q^m}$ and $\alpha\in \F_{q^t}$. 
Then for each $\mu, \eta \in \F_q$ we have 
\[(\mu x + \eta y)^d = \sum_{i=0}^d \binom{d}{i} (\mu \alpha f_1)^i (\eta \alpha f_2)^{d-i}= \alpha^d \sum_{i=0}^d \gamma_i f_i',\]
where $\gamma_i \in \F_q$ and $f_i'\in \F_{q^m}$ for $i=0,1,\ldots, d$, and hence 
$(\mu x + \eta y)^d=\alpha^d f$ for some $f\in \F_{q^m}$ (depending on $\mu$ and $\eta$). 
Let $\ell'=\la f_1, f_2 \ra_q \subseteq \PG_q(\F_{q^m})$. Then $\lambda_{\alpha^d}$ is a projectivity which maps 
$\ell'^d$ into $\ell^d$. It is easy to see that $o(\ell)=o(\ell')$. 
\end{proof}

\begin{proposition}
\label{projeq} 
Let $\ell$ be a line of $\PG_q(\F_{q^t})$ and let $P$ and $Q$ be any two different points of $\ell$. There exists a projectivity $\mu$ of 
$\PG_q(\F_{q^t})$ and a line $\ell'=\la 1,z \ra_q$ such that $\mu(P^d)=\la 1 \ra_q$, $\mu(Q^d)=\la z^d \ra_q$, $\mu(\ell^d)=\ell'^d$ and $o(\ell)=o(\ell')$.  
\end{proposition}
\begin{proof}
Let $P=\la x \ra_q$, $Q=\la y \ra_q$ and let $z=y/x$. Then 
$o(\ell)=o(\ell')$ trivially holds. 
We have $\ell=\{\la x w \ra_q \mid \la w \ra_q \in \ell'\}$, thus $\ell^d= \{ \la x^d u \ra_q \mid \la u \ra_q \in \ell'^d \}$. 
We have $P^d=\lambda_{x^d}(\la 1 \ra_q)$, $Q^d=\lambda_{x^d}(\la z^d \ra_q)$ and $\ell^d=\lambda_{x^d}(\ell'^d)$. 
Let $\mu$ be the inverse of $\lambda_{x^d}$. 
\end{proof}

\begin{proposition}
\label{num}
Let $\nu\geq 1$ and $h\geq 0$ be integers such that $\gcd(\nu,t)=1$. Also, let $m$ be a positive divisor of $t$. 
Then $\theta_{\nu-1}$ is invertible modulo $\theta_{t-1}$, denote its inverse by $\theta_{\nu-1}^{-1}$. 
Also, denote\footnote{In this section $\theta_{\nu-1}^{-1}$ and $\nu^{-1}$ will always denote inverses
modulo $\theta_{t-1}$ and $m$, respectively.}
the inverse of $\nu$ modulo $m$ by $\nu^{-1}$. Then we have 
\begin{equation}
\label{congr}
\theta_{\nu-1}^{-1}\theta_{h-1} \equiv \frac{q^{\nu n}-1}{q^{\nu}-1} \pmod {\theta_{m-1}},
\end{equation}
for each  $n$, such that $n \equiv h\nu^{-1} \pmod m$.
\end{proposition}
\begin{proof}
The following congruence relation is equivalent to \eqref{congr},
\begin{equation}
\label{congr2}
\theta_{h-1} \equiv \frac{q^{\nu n}-1}{q-1} \pmod {\theta_{m-1}}.
\end{equation}
Because of the choice of $n$, we have $\nu n \equiv h \pmod m$, say 
$\nu n=Km+h$. As $q^m=(q-1)\theta_{m-1}+1$, it follows that $q^{Km+h}-1 \equiv (q^m)^Kq^h-1 \equiv q^h-1 \pmod {\theta_{m-1}}$.
\end{proof}

\begin{lemma}
\label{general}
Take a line $\ell$ of $\PG_q(\F_{q^m})$ with $o(\ell)=m$ and let $\nu\geq 1$, $\gcd(\nu,m)=1$, and $1 \leq n \leq m-1$ be integers.
Let
\[d=\frac{q^{n\nu}-1}{q^\nu-1}.\] 
Then any $n+1$ points of $\ell^d$ are in general position, that is, they span an $n$-dimensional subspace of $\PG_q(\F_{q^m})$.
\end{lemma}
\begin{proof}
We proceed by induction on $n$. If $n=1$, then $d=1$ and hence $\ell^d=\ell$, in which case the assertion is trivial. 
Now let $n>1$ and suppose to the contrary that $\ell$ contains $n+1$ points, $Q_0,Q_1,\ldots,Q_n$, such that 
$Q_0^d \in \la Q_1^d, \ldots, Q_{n}^d\ra$. 
According to Proposition \ref{projeq} there is a line $\ell'$ such that $o(\ell)=o(\ell')$, $\ell'=\la 1, z\ra_q$ and 
\begin{equation}
\label{span}
\la z^d \ra_q \in \la Q_1'^d, \ldots, Q_{n-1}'^d, \la 1 \ra_q \ra,
\end{equation}
where $Q_i' \in \ell'$, for $i=1,\ldots,n-1$. 
It follows from \eqref{span} that there exist $\alpha_i \in \F_q$ for $i=1,2,\ldots, n$ and $\lambda_i\in \F_q^*$ for $i=1,2,\ldots, n-1$ 
such that $\lambda_j\neq \lambda_k$ for $j\neq k$ and
\begin{equation}
\label{eq1}
z^d=\alpha_{n}+\sum_{i=1}^{n-1} \alpha_i (1+\lambda_i z)^d.
\end{equation}
Now we use $d=1+q^{\nu}+\ldots+q^{(n-1)\nu}$ and we take $q^{\nu}$-th powers of both sides to obtain
\begin{equation}
\label{eq2}
z^{q^{\nu}+\ldots+q^{n\nu}}=\alpha_{n}+\sum_{i=1}^{n-1} \alpha_i (1+\lambda_i z)^{q^{\nu}+\ldots+q^{n\nu}}.
\end{equation}
Subtracting \eqref{eq1} from \eqref{eq2} yields
\[z^{q^{\nu}+\ldots+q^{(n-1)\nu}}\left(z^{q^{n\nu}}-z\right)=\]
\[\sum_{i=1}^{n-1} \alpha_i (1+\lambda_i z)^{q^{\nu}+\ldots+q^{(n-1)\nu}}\left((1+\lambda_i z)^{q^{n\nu}}-(1+\lambda_i z)\right).\]
We have $(1+\lambda_i z)^{q^{n\nu}}-(1+\lambda_i z)=\lambda_i\left(z^{q^{n\nu}}-z\right)$, which is non-zero because $o(z)=o(\ell)=m$ does not divide $n\nu$. So we can divide both sides by $z^{q^{n\nu}}-z$ and take $q^{\nu}$-th roots to obtain
\[z^{1+q^{\nu}+\ldots+q^{(n-2)\nu}}=\sum_{i=1}^{n-1} \alpha_i\lambda_i (1+\lambda_i z)^{1+q^{\nu}+\ldots+q^{(n-2)\nu}}.\]
According to the induction hypothesis any $n$ points of $\ell'^{1+q^{\nu}+\ldots+q^{(n-2)}\nu}$ are in general position, a contradiction.
\end{proof}

\begin{corollary}
\label{cor}
Take a line $\ell$ of $\PG_q(\F_{q^t})$ and let $o(\ell)=m$. 
Take positive integers $\nu, h\in\{1,2,\ldots,t-1\}$ such that $\gcd(\nu,t)=1$.
Suppose that $m$ does not divide $h$ and let $n\equiv h\nu^{-1} \pmod m$ such that $1\leq n \leq m-1$. 
Let $d={\theta_{\nu-1}^{-1}\theta_{h-1}}$. If $q \geq n$, then $\ell^d$ contains $n+1$ points in general position. 
\end{corollary}
\begin{proof}
Let $d'=\frac{q^{\nu n}-1}{q^{\nu}-1}$. 
Proposition \ref{num} yields $d \equiv d' \pmod {\theta_{m-1}}$, thus Proposition 
\ref{prop 3} implies that $\ell^d$ and $\ell^{d'}$ are projectively equivalent. Also, Proposition \ref{prop 2} yields that 
$\ell^{d'}$ is projectively equivalent to $\ell'^{d'}$ for some line $\ell' \in \PG_q(\F_{q^m})$ with $o(\ell')=m$. 
Then the assertion follows from Lemma \ref{general}. 
\end{proof}

\begin{lemma}
\label{span0}
Let $\ell$ be a line of $\PG_q(\F_{q^t})$ and let $d$ be a non-negative integer. 
Then $d$ can be written uniquely as $d=\sum_{i=0}^{\infty}a_i q^i$ ($0\le a_i<q$, $i=0,1,\ldots$). 
Let $d_q=\sum_{i=0}^{\infty}a_i$ and suppose $q\geq d_q$. Also suppose that in $\ell^d$ there exists a set of $d_q+1$ points in general position. Denote by $\cB$ the $d_q$-subspace of $\PG_q(\F_{q^t})$ spanned by this $(d_q+1)$-set. Then $\ell^d$ is a normal rational curve of order $d_q$ in $\cB$.
\end{lemma}
\begin{proof}
If $d=d_q=0$, then the statement is trivial. Otherwise, according to Proposition \ref{projeq}, we may assume $\ell=\la 1,z \ra_q$, where $\la z^d \ra_q$ is a point of a $(d_q+1)$-set whose elements span $\cB$. Denote these $d_q+1$ points by $\{\la(1+\lambda_1 z)^d\ra_q,\la(1+\lambda_2 z)^d\ra_q,\ldots,\la(1+\lambda_{d_q} z)^d\ra_q, \la z^d \ra_q\}$.
Fix $\{(1+\lambda_1 z)^d,(1+\lambda_2 z)^d,\ldots,(1+\lambda_{d_q} z)^d, z^d\}$ as a basis of the corresponding $(d_q+1)$-dimensional vector subspace of $\F_{q^t}$. To prove $\ell^d \subseteq \cB$ we should find for each $\lambda \in \F_q$ elements $\alpha_1,\ldots,\alpha_{d_q},\alpha \in \F_q$, considered as coordinates of the points of $\ell^d$, such that
\begin{equation}
\label{eq0}
(1+\lambda z)^d=\sum_{i=1}^{d_q} \alpha_i(1+\lambda_i z)^d + \alpha z^d.
\end{equation}
Note that for $\mu \in \F_q$ we have
\[(1+\mu z)^d=\prod_{i=0}^{\infty}(1+\mu z^{q^i})^{a_i}=\prod_{i=0}^{\infty}\sum_{j=0}^{a_i}z^{jq^i}\mu^j \binom{a_i}{j}=\sum_{i=0}^{d}z^if_i(\mu),\]
where $f_i$ is a polynomial over $\F_q$ of degree at most $d_q$. 
Note that $\deg\, f_i =d_q$ if and only if $i=d$. Also, $f_d(\mu)=\mu^{d_q}$. 
Thus \eqref{eq0} can be written as
\[\sum_{i=0}^{d}z^if_i(\lambda)=\left(\sum_{j=1}^{d_q} \alpha_j \sum_{i=0}^{d}z^if_i(\lambda_j)\right) + \alpha z^d,\]
\begin{equation}
\label{eki}
\sum_{i=0}^{d-1}z^i\left(f_i(\lambda)-\sum_{j=1}^{d_q} \alpha_j f_i(\lambda_j)\right) = z^d \left(\sum_{j=1}^{d_q} \alpha_j \lambda_j^{d_q} + \alpha - \lambda^{d_q}\right).
\end{equation}
If $f_i(X)\in \F_{q}[X]$ is $\sum_{k=0}^{d_q-1}a_{ik}X^k$ ($i=0,1,\ldots,d-1$), then 
\begin{equation}
\label{eki2}
f_i(\lambda)-\sum_{j=1}^{d_q} \alpha_j f_i(\lambda_j)=\sum_{k=0}^{d_q-1}a_{ik}(\lambda^k-\sum_{j=1}^{d_q}\alpha_j\lambda_j^k).
\end{equation}
We show that there exist $\alpha, \alpha_1,\ldots,\alpha_{d_q}\in \F_q$, such that $\lambda^k-\sum_{j=1}^{d_q}\alpha_j\lambda_j^k=0$ 
for each $k=0,\ldots, d_q-1$ and $\sum_{j=1}^{d_q} \alpha_j \lambda_j^{d_q} + \alpha - \lambda^{d_q}=0$. 
Then $\eqref{eki}$ and hence \eqref{eq0} have a solution. 
Consider the following matrix equation
\[\left[ \begin{array}{c} 1 \\ \lambda \\ \vdots \\ \lambda^{d_q} \end{array} \right] =
 \begin{bmatrix} 0 & 1 & \ldots & 1 \\ 0 & \lambda_1 & \ldots & \lambda_{d_q} \\ \vdots & \vdots & \vdots & \vdots \\ 
1 & \lambda_1^{d_q} & \ldots & \lambda_{d_q}^{d_q} \end{bmatrix} 
\left[ \begin{array}{c} \alpha \\ \alpha_1 \\ \vdots \\ \alpha_{d_q} \end{array} \right].\]
Denote the $(d_q+1)\times (d_q+1)$ matrix on the right-hand side by $M$ and denote by $N$ the $d_q\times d_q$ matrix obtained from $M$ by removing its first column and its last row. As $\det (M)=\pm \det(N)$ and $N$ is a Vandermonde matrix,
it follows from $\lambda_i \neq \lambda_j$ for $i\neq j$ that $\det (M)\neq 0$. 
Thus the system of equations has a unique solution. 
Let $\tau$ be the projectivity of $\cB$ whose matrix is $M^{-1}$. Then for each $\lambda\in \F_q$ we have 
$\la(1,\lambda,\ldots,\lambda^{d_q})\ra_q^{\tau}
=\la(1+\lambda z)^d\ra_q$ and $\la(0,0,\ldots,0,1)\ra_q^{\tau}=\la z^d \ra_q$. 
\end{proof}

\begin{theorem}
\label{maintheta}
Let $\ell$ be  a line of $\PG_q(\F_{q^t})$ and let $o(\ell)=m$. 
Take positive integers $\nu, h\in\{1,2,\ldots,t-1\}$ such that $\gcd(\nu,t)=1$.
Suppose that $m$ does not divide $h$ and let $n\equiv h\nu^{-1} \pmod m$ such that $1\leq n \leq m-1$. 
Let $d={\theta_{\nu-1}^{-1}\theta_{h-1}}$. If $q \geq n$, then $\ell^d$ is a normal rational curve of 
order $n$ in some $n$-subspace of 
$\PG_q(\F_{q^t})$. 
\end{theorem}
\begin{proof}
Corollary \ref{cor} yields the existence of $n+1$ points of $\ell^d$ in general position. 
According to Propositions \ref{prop 3} and \ref{num}, $\ell^d$ is projectively equivalent to $\ell^{d'}$ with $d'=1+q^{\nu}+\ldots+q^{(n-1)\nu}$. 
The assertion follows from Lemma \ref{span0} since $d'_q=n$. 
\end{proof}

\begin{corollary}\emph{\cite{FaKiMaPa2002,Ha1983,He1995,LaZa2014, Ma2014}}
\label{manyaouthors}
If $\ell$ is a line of $\PG_q(\F_{q^t})$, $q+1\geq t$, then $\ell^{-1}$ is 
a normal rational curve in some $(m-1)$-space of 
$\PG_q(\F_{q^t})$ such that $m$ divides $t$.
\end{corollary}
\begin{proof}
The $\PG(t-1,q) \rightarrow \PG(t-1,q)$ map $\la x \ra_q \mapsto \la x^{-1} \ra_q$ is the composition of 
the maps $\la x \ra_q\mapsto \la x^{\theta_{t-2}} \ra_q $ and $\la x \ra_q \mapsto \la x^q \ra_q$. The latter one is a projectivity of 
$\PG_q(\F_{q^t})$ and hence for each line $\ell \subset \PG_q(\F_{q^t})$, $\ell^{-1}$ is projectively equivalent to 
$\ell^{\theta_{t-2}}$. Let $o(\ell)=m$. Since $m$ divides $t$, it cannot divide $t-1$ and hence Theorem \ref{maintheta} with $h=t-1$ and $\nu=1$ yields that $\ell^{-1}$ is a normal rational curve in some $m-1$ space of $\PG_q(q^t)$.
\end{proof}
As a corollary of Theorem  \ref{maintheta} it holds:
\begin{theorem}\label{t:tutte-NRC}
  Let $\Lb=p_{\Gamma,\ell_0}(\Sigma)$ be a scattered linear set of pseudoregulus type in 
  $\ell_0\cong\PG(1,q^t)$, $q\ge t$, and let $r$ be a $q$-order subline of $\ell_0$ contained in $\Lb$.
  Then $\Sigma\cap p^{-1}_{\Gamma,\ell_0}(r)$ is a normal rational curve in some $n$-subspace of $\Sigma$.
  More precisely, assuming Proposition \ref{p:inverse}, a divisor $m>1$ of $t$ exists such that 
  $\Sigma\cap p^{-1}_{\Gamma,\ell_0}(r)$ is 
	a normal rational curve of order $n$ in some $n$-subspace of $\Sigma$, where 
	$1\leq n \leq m-1$ and $n \equiv h\nu^{-1} \pmod m$.
\end{theorem}

Since the normal rational curves dealt with in Theorem \ref{t:subl-v-carriers} are of order $t-1$,
as a consequence of Theorem \ref{t:tutte-NRC} one obtains:

\begin{theorem}\label{p:unique-family}
  Let $t$ be a prime and $q \geq t$.
  For any $h\in\{1,2,\ldots,t-1\}$ and any line $f$ contained in an element of $\cS_h$ (see (\ref{e:def-Sh})),
  the $q$-order subline $r=\cB(f)$ is the projection under $p_{\Gamma,\ell_0}$  of a normal rational curve
  of order $n\equiv h\nu^{-1}\pmod t$ contained in $\Sigma$.
  In particular, if $q>t$, the normal rational curves related to lines contained in elements 
  of $\cS_{-\nu}$ are 
  exactly the normal rational curves containing the points $P_{\Gamma}^{\hs^i}$, 
  $i=0,1,\ldots,t-1$, constructed in Theorem \ref{t:subl-v-carriers}.
\end{theorem}
\begin{proof}  
Since $t$ is a prime, we have $o(f^{\iota^{-1}})=t$. 
According to Theorem \ref{t:tutte-NRC} the order of the normal rational curve 
$\Sigma\cap p^{-1}_{\Gamma,\ell_0}(r)$ is $n$ if and only if $n \equiv h\nu^{-1} \pmod t$, 
i.e.\ when $h \equiv n\nu \pmod t$. 

The number of normal rational curves constructed in 
Theorem \ref{t:subl-v-carriers} is $\theta_{t-1}\theta_{t-2}/\theta_1$, which is the same as the number of reguli $\cF(\ell)$, $\ell$ is a 
$q$-subline of $\Lb$, such that there is a transversal of $\cF(\ell)$ contained in an element of $\cS_h$.
\end{proof}

\begin{remark}
  Theorem \ref{p:unique-family} extends \emph{\cite[Theorem 5.2]{BaJa2014}} and \emph{\cite[Lemma 18]{LaVa2010}}.
\end{remark}

\section*{Acknowledgements}
This research was supported by the Italian
Ministry of Education, University and Research (PRIN 2012 project ``Strutture geometriche, combinatoria e loro
applicazioni'').

\vspace{5mm}

\begin{flushright}
Authors' Address:
Dipartimento di Tecnica e\\
Gestione dei Sistemi Industriali,\\
Universit\`a di Padova,\\
Stradella S. Nicola, 3,\\
I-36100 Vicenza,\\
Italy
\end{flushright}

\end{document}